\documentclass[11pt]{article}

\usepackage{xcolor}

\usepackage{amsfonts}
\usepackage{amsthm}
\usepackage{amssymb, amsmath}
\usepackage[english]{babel}
\usepackage{comment}

\usepackage{hyperref}

\newtheorem{theo}{Theorem}[section]
\newtheorem{cor}[theo]{Corollary}
\newtheorem{lemma}[theo]{Lemma}

\theoremstyle{definition}

\newtheorem{remark}[theo]{Remark}

\numberwithin{equation}{section}

\newcommand{\R}{\mathbb{R}}
\newcommand{\N}{\mathbb{N}}

\newcommand{\Sc}{\mathcal{S}}
\renewcommand{\dim}{\mathrm{dim}}

\newcommand{\vol}{\mathrm{vol}}

\newcommand{\rank}{\mathrm{rank}}
\DeclareMathOperator{\id}{id}
\DeclareMathOperator{\dis}{d}
\DeclareMathOperator{\diag}{diag}
\DeclareMathOperator{\diam}{diam}
\DeclareMathOperator{\codi}{codim}
\newcommand{\hr}{ \hookrightarrow }

\newcommand{\vertiii}[1]{{\left\vert\kern-0.25ex\left\vert\kern-0.25ex\left\vert #1 
    \right\vert\kern-0.25ex\right\vert\kern-0.25ex\right\vert}}

\begin{document}
\title{Entropy numbers of embeddings of Schatten classes}
\author{Aicke Hinrichs\footnote{Institute of Analysis, University Linz, Altenberger Str. 69, 4040 Linz, Austria, {\tt aicke.hinrichs@jku.at}},
Joscha Prochno\footnote{School of Mathematics \& Physical Sciences, University of Hull, Cottingham Road, HU6 7RX Hull, United Kingdom,
{\tt j.prochno@hull.ac.uk}},
Jan Vyb\'\i ral\footnote{Department of Mathematical Analysis, Charles University, Sokolovsk\'a 83, 186 00, Prague 8, Czech Republic, {\tt vybiral@karlin.mff.cuni.cz};
this author was supported by the ERC CZ grant LL1203 of the Czech Ministry of Education and by the Neuron Fund for Support of Science.}}
\date{\today}
\maketitle

\begin{abstract}
Let $0<p,q \leq \infty$ and denote by $\Sc_p^N$ and $\Sc_q^N$ the corresponding finite-dimensional Schatten classes. We prove optimal bounds, up to constants only depending on $p$ and $q$, for the entropy numbers of natural embeddings between $\Sc_p^N$ and $\Sc_q^N$. This complements the known results in the classical setting of natural embeddings between finite-dimensional $\ell_p$ spaces due to Sch\"utt, Edmunds-Triebel, Triebel and Gu\'edon-Litvak/K\"uhn.
We present a rather short proof that uses all the known techniques as well as a constructive proof of the upper bound in the range $N\leq n\leq N^2$ that allows deeper structural insight and is therefore interesting in its own right. Our main result can also be used to provide an alternative proof of recent lower bounds in the area of low-rank matrix recovery.
\end{abstract}


\section{Introduction and Main Results}

The probably most important class of unitary operator ideals are the Schatten $p$-classes $\Sc_p$ ($0<p \leq \infty$) consisting of all compact operators between Hilbert spaces for which the sequence of their singular values belongs to $\ell_p$. In many ways, the structure of these spaces resembles that of the classical sequence spaces $\ell_p$, and the Schatten classes $\Sc_p$ form, in the sense of multiplication, their non-commutative counterpart. Some classical results on the topological structure of $\Sc_p$ spaces can be found, for instance, in \cite{McC1967} and \cite{AL1975}. Recent years have seen increased interest in those non-commutative $L_p$ spaces. One such case, in part motivating our work, lies in the area of compressed sensing in the form of low-rank matrix recovery (see, e.g., \cite{CR2009,CT2010,CK} and the references therein.)

The study of bounded linear operators between Banach spaces is a classical part of functional analysis. Of particular interest are compact operators, that is, those operators for which the image of the closed unit ball under the operator is relatively compact in the codomain space. In general, the extent to which a bounded linear operator is compact is measured by its sequence of entropy numbers.
The $n$-th (dyadic) entropy number of a bounded linear operator $T$ between Banach spaces $X$ and $Y$ (denoted by $e_n(T)$ in the sequel)
is the infimum over all $\varepsilon>0$ such that there are $y_1,\dots,y_{2^{n-1}}\in Y$ with 
\[
T(B_X) \subseteq \bigcup_{j=1}^{2^{n-1}}(y_j+\varepsilon B_Y),
\] 
where $B_X$, $B_Y$ denote the closed unit balls of $X$ and $Y$, respectively.
In fact, it is true that an operator between Banach spaces is compact if and only if the sequence of its entropy numbers converges to $0$. The theory of entropy numbers and their systematic study was introduced by Pietsch in \cite{Pietsch}. 

In the case of symmetric Banach sequence spaces, entropy numbers of diagonal operators were studied by Sch\"utt in 1984 \cite{S}, where he obtained bounds of exact order up to a logarithmic factor. In the particular case of the natural identity between $\ell_p^N$ and $\ell_q^N$ in the Banach space regime $1\leq p<q \leq \infty$, we write $\ell_p^N\hookrightarrow\ell_q^N$, he obtained the following asymptotically sharp estimates for the $n$-th dyadic entropy number $e_n$ \cite[Theorem 1]{S}: 
\vskip 1mm
If $1\leq p<q \leq \infty$ and $n,N\in\N$, then 
\begin{equation}\label{eq:scalar}
e_n\big(\ell_p^N\hookrightarrow\ell_q^N\big) \asymp_{p,q}
\begin{cases}
1\,,& 1\leq n \leq \log N\,; \\
\Big(\frac{\log(N/n+1)}{n}\Big)^{1/p-1/q}\,,& \log N \leq n \leq N\,; \\
2^{-n/N}\cdot N^{1/q-1/p}\,,& N\leq n.
\end{cases}
\end{equation}
It was in 1996 that Edmunds and Triebel proved corresponding upper bounds also for the quasi-Banach space regime $0<p \leq q \leq \infty$ \cite[Proposition 3.2.2]{ET}. Shortly after, complementing and corresponding lower estimates for $1\leq n \leq \log N$ and $N\leq n$ were obtained by Triebel in \cite[Theorem 7.3]{T97}.
Finally, the gap on the medium scale, $\log N \leq n \leq N$, was closed by Gu\'edon and Litvak \cite{GL} and, independently, by K\"uhn \cite{K2001}.

The main subject of this paper
is to prove the non-commutative counterpart
of \eqref{eq:scalar}, that is, the computation of entropy numbers of the natural identities between finite-dimensional Schatten classes in the full range $0<p,q \leq \infty$. 
For $0<p\le \infty$, let $\Sc_p^N$ denote the $N^2$-dimensional space of all real $N \times N$ matrices $A$ equipped with the Schatten $p$-norm
\[
\|A\|_{\Sc^N_p}=\bigg(\sum_{j=1}^N\sigma_j(A)^p\bigg)^{1/p},
\]
where $\sigma_j(A)$, $j=1,\dots,N$ are the singular values of $A$.
For $0<p,q \le \infty$, we denote the natural identity map from $\Sc_p^N$ to $\Sc_q^N$ by $\Sc_p^N \hr \Sc_q^N$.

Our main result is the following theorem.

\begin{theo}\label{thm:main}
 	Let $n,N\in\N$ and $0<p,q \leq \infty$. Then,
	\[
	e_n\big( \Sc_p^N \hr \Sc_q^N \big) \asymp_{p,q}
	\begin{cases}
	  2^{-\frac{n}{N^2}}\cdot N^{1/q-1/p}, & 0<q\le p \le  \infty\,; \\
	  1,                                   & 0<p\leq q \leq \infty, 1\leq n \leq N\,; \\
	  \big(\frac{N}{n}\big)^{1/p-1/q},     & 0<p\leq q \leq \infty, N \leq n \leq N^2\,; \\
	2^{-\frac{n}{N^2}}\cdot N^{1/q-1/p},   & 0<p\leq q \leq \infty, N^2\leq n,
	\end{cases}
	\]
	where $\asymp_{p,q}$ denotes equivalence up to constants depending only on $p$ and $q$.
\end{theo}

The following remark briefly explains the connection of Theorem \ref{thm:main} to low-rank matrix recovery.

\begin{remark}
Our main result can be used to provide an alternative proof of lower bounds in the area of low-rank matrix recovery.
Let us assume that we want to recover a matrix $X$ from the unit ball of $\Sc_p^N$ (which we denote by $B_p^N$ in the sequel) from $m$ linear measurements of $X$ represented by
the information map $${\mathcal A}(X)=(\langle A_1,X\rangle,\dots,\langle A_m,X\rangle)\in\R^m.$$
The search for the optimal pair of information map ${\mathcal A}:\R^{N\times N}\to \R^m$ and recovery map $\Delta:\R^m\to \R^{N\times N}$
is expressed in terms of the quantities
\[
E_m\big(B_p^N,\Sc_q^N\big)=\inf_{\substack{{\mathcal A}\,:\,\R^{N\times N}\to \R^m\\ \Delta\,:\,\R^m\to \R^{N\times N}}}\,\sup_{X\in B_p^N}\|X-\Delta({\mathcal A}(X))\|_{\Sc_q^N}.
\]
Using their relation to Gelfand numbers and Carl's inequality (see Section \ref{s:2} and \cite[Section 10]{FR} for details), Theorem \ref{thm:main}
immediately implies that, for $0<p\le q\le \infty$ and $1\le m\le N^2$, 
\[
E_m\big(B_p^N,\Sc_q^N\big)\gtrsim_{p,q}\min\Bigl\{1,\frac{N}{m}\Bigr\}^{1/p-1/q}.
\]
This provides another proof for the lower bound recently given in \cite[Theorem 5.5]{CK}, where the same result was obtained in the regime
$0<p\le 1$ and $p<q\le 2$.
\end{remark}

As typical in the theory of Banach spaces, our proofs reflect a variety of ideas and tools of analytical and geometric flavor. These include 

\renewcommand\labelitemi{\tiny$\bullet$}
\begin{itemize}
\item Carl's inequality for the case of quasi-Banach spaces \cite{HKV} 
\item two-sided asymptotic estimates for the Gelfand numbers of identities of Schatten classes \cite{CK}
\item volume estimates for metric balls in the Grassmannian going back to the work of Szarek \cite{Sz} (see also \cite{PV}), and
\item interpolation estimates of entropy numbers in the domain and target spaces (see \cite{Pi71} and \cite{PiXu}).
\end{itemize}

The rest of the paper is organized as follows. In the next section, we give the necessary definitions and notation and recall some results needed for the proof of Theorem \ref{thm:main}.
In Section \ref{s:3} a proof of Theorem \ref{thm:main}, which is intended to be as short as possible using all known machinery to estimate entropy numbers, is presented.
Most of the cases are rather standard, but some require deeper tools from the structure theory of Schatten classes.
For some purposes it is much more instructive to find explicit covers achieving optimal upper bounds for the entropy numbers. This construction will be given in Section \ref{s:4}.  

\section{Notation and Preliminaries} \label{s:2}
In this section we present the notation and preliminaries we use throughout this text. For the sake of clarity, we subdivide the background material by topic. Since we have a rather broad audience in mind, we try to keep this work as self contained as possible.

\subsection{General notation}

The cardinality of a set $A$ will be denoted by $|A|$. We denote by $\vol_n(A)$ the $n$-dimensional Lebesgue measure of a Borel measurable set $A\subseteq\R^n$ and will usually suppress the index $n$ for brevity.
We supply the $n$-dimensional Euclidean space $\R^n$ with its standard inner product $\langle\,\cdot\,,\,\cdot\,\rangle$ and corresponding Euclidean structure $\|\,\cdot\,\|_2$.
For $0< p\leq \infty$, we denote by $\ell_p^n$ the space $\R^n$ equipped with the (quasi-)norm 
\[
\|(x_i)_{i=1}^n\|_p=\begin{cases}
\big(\sum_{i=1}^n|x_i|^p\big)^{1/p}\,,& 0< p < \infty\,; \\
\max_{1\leq i \leq n}|x_i|\,,& p=\infty.
\end{cases}
\]
Given two (quasi-)Banach spaces $X,Y$, we denote by $B_X$ the closed unit ball of $X$ and by $\mathcal L(X,Y)$ the space of bounded linear operators between $X$ and $Y$ equipped with the standard operator (quasi-)norm.

Finally, for two sequences $(a(n))_{n\in\N}$ and $(b(n))_{n\in\N}$ of non-negative real numbers, we write $a(n)\asymp b(n)$ provided that there exist constants $c,C\in(0,\infty)$ such that $cb(n)\leq a(n)\leq Cb(n)$ for all $n\in\N$. If the constants depend on some parameter $p$, we shall write $a(n)\lesssim_p b(n)$, $a(n)\gtrsim_p b(n)$ or, if both hold, $a(n)\asymp_p b(n)$.

\subsection{Quasi-Banach spaces}

We briefly present some background on quasi-Banach spaces. If $X$ is a (real) vector space, then we call $\|\cdot\|_X:X\to[0,\infty)$ a quasi-norm if 
\begin{enumerate}
\item $\|x\|_X=0$ if and only if $x=0$
\item $\|\alpha x\|_X=|\alpha|\cdot \|x\|_X$ for all $\alpha\in\R$ and $x\in X$
\item There exists a constant $C\in[1,\infty)$ such that $\|x+y\|_X \leq C(\|x\|_X+\|y\|_X)$ for all $x,y\in X$.
\end{enumerate}
By the Aoki-Rolewicz theorem (see \cite{Aoki,Rol}), every quasi-norm is equivalent to some $p$-norm, that is, there exists a mapping
$|||\cdot|||_X:X\to[0,\infty)$, which is equivalent to $\|\cdot\|_X$ on $X$, and some $0<p\leq 1$ such that $|||\cdot|||_X$ satisfies 1. and 2. above, while 3. is replaced by
$|||x+y|||_X^p \leq |||x|||_X^p+|||y|||_X^p$. The mapping $|||\cdot|||_X$ is then
called a $p$-norm and if $X$ is complete with respect to the metric induced by $|||\cdot|||_X^p$, $X$ is called a $p$-Banach space and the original space $X$ with the quasi-norm $\|\cdot\|_X$ is called a quasi-Banach space. Therefore, as far as estimates up to a multiplicative constant are concerned, we may assume that a quasi-Banach space is equipped with an equivalent $p$-norm instead.

\subsection{Singular values and Schatten $p$-classes}

The singular values $\sigma_1,\dots,\sigma_N$ of a real $N\times N$ matrix $A$ are defined to be the square roots of the eigenvalues of the positive self-adjoint operator $A^*A$, which are simply the eigenvalues of $|A|:=\sqrt{A^*A}$. The singular values are arranged in non-increasing order, that is, $\sigma_1(A) \geq \dots \geq \sigma_N(A)\geq 0$. The singular value decomposition shall be used in the form $A=U\Sigma V^T$, where $U,V\in\R^{N\times N}$ are orthogonal matrices,
and $\Sigma\in\R^{N\times N}$ is a diagonal matrix with $\sigma_1(A),\dots,\sigma_N(A)$ on the diagonal.

For $0<p\le \infty$, the Schatten $p$-class $\Sc_p^N$ is the $N^2$-dimensional space of all $N\times N$ real matrices acting from $\ell_2^N$ to $\ell_2^N$ equipped with the Schatten $p$-norm
\[
\|A\|_{\Sc^N_p}=\bigg(\sum_{j=1}^N\sigma_j(A)^p\bigg)^{1/p}.
\]
Let us remark that $\|\cdot\|_{\Sc_1^N}$ is the nuclear norm, $\| \cdot \|_{\Sc_2^N}$ the Hilbert-Schmidt norm, and $\|\cdot\|_{\Sc_\infty^N}$ the operator norm. We denote the unit ball of $\Sc_p^N$ by
\begin{align*}
B_p^N :=\Big\{A\in\R^{N\times N}\,:\,\|A\|_{\Sc^N_p}\le 1\Big\}.
\end{align*}
The following volume estimate for the unit balls of the Schatten $p$-classes can be found in \cite[Corollary 8]{SR} for the case $p\ge 1$. Following along the lines of the proof in \cite{SR} it is not too hard to see that it carries over to the case $0<p\leq 1$.

\begin{lemma}\label{lem:schatten-volume}
Let $0<p\le \infty$ and $N\in\N$. Then
 \begin{equation}\label{eq:vol_BpN}
  \vol(B_p^N)^{1/N^2}\asymp_p N^{-1/2-1/p}.
 \end{equation}
\end{lemma}

Please note that if $p\ge 1$, then $\|\cdot\|_{\Sc_p^N}$ is a norm, and if $0<p<1$, then $\|\cdot\|_{\Sc_p^N}$ is a $p$-norm, i.e., a quasi-norm satisfying the $p$-triangle inequality
\[
\|A+B\|_{\Sc_p^N}^p \leq \|A\|_{\Sc_p^N}^p + \|B\|_{\Sc_p^N}^p,\qquad A,B\in \Sc_p^N.
\]
For $p,q\in[1,\infty]$ with $\frac{1}{p}+\frac{1}{q}=1$, the Schatten classes satisfy the following H\"older type inequality,
\[
\|AB\|_{\Sc_1^N} \leq \|A\|_{\Sc_p^N}\|B\|_{\Sc_q^N}, \qquad A\in \Sc_p^N, B\in \Sc_q^N.
\]

\subsection{Unitarily invariant and symmetric norms}

The Schatten $p$-norms (for $p\ge 1$) are examples of unitarily invariant norms. In particular, unitarily invariant norms correspond to 1-symmetric norms, similarly as the Schatten $p$-norm corresponds to the $\ell_p$-norm.

A basis $\{e_i\}_{i=1}^N$ of an $N$-dimensional Banach space $X$ is called $1$-symmetric if and only if for all permutations $\pi$ of $\{1,\dots,N\}$, all signs $\varepsilon_i\in\{-1,1\}$, $i=1,\dots,N$, and all $a_1,\dots,a_N\in\R$,
\[
\Big\|\sum_{i=1}^na_ie_i\Big\|_X = \Big\|\sum_{i=1}^N \varepsilon_ia_{\pi(i)}e_i\Big\|_X.
\]

A norm on $\mathcal L(\ell_2^N,\ell_2^N)$ is called a unitarily invariant norm, we write $\|\cdot\|_{inv}$, if it satisfies 
\[
\|A\|_{inv}=\|UAV\|_{inv}
\] 
for any $A\in\mathcal L(\ell_2^N,\ell_2^N)$ and all isometries $U,V$ on $\ell_2^N$. Since we are only considering the real case, $U,V$ are simply orthogonal matrices. It is well-known and easy to see that to any such norm $\|\cdot\|_{inv}$ corresponds a $1$-symmetric norm $\|\cdot \|$ on $\R^N$ so that, for all $A\in\mathcal L(\ell_2^N,\ell_2^N)$,
\[
\|A\|_{inv} = \|(\sigma_i(A))_{i=1}^N\|,
\]
where $(\sigma_i(A))_{i=1}^N$ is the vector of singular values of $A$,
and, on the other hand, $\|A\|_{inv}=\|(x_1,\dots,x_N)\|$ for any diagonal matrix $A=\diag(x_1,\dots,x_N)$. For some background on unitarily invariant norms, we refer the reader to \cite{HJ}.

\subsection{Geometry of Grassmannians}

Given $N\in\N$ and $k\in\{0,1,\ldots,N\}$, we denote by $G_{N,k}$ the Grassmannian of $k$-dimensional linear subspaces of $\R^N$. Recall that its dimension is $k(N-k)$.
We equip $G_{N,k}$ with a metric $\dis(\cdot,\cdot)$ induced by some unitarily invariant norm $\| \, \cdot \,\|_{inv}$ on $\mathcal L(\ell_2^N,\ell_2^N)$ by defining
\begin{equation}{\label{eq:operator norm projection}}
\dis(E,F):= \|P_E-P_F\|_{inv},\qquad E,F\in G_{N,k},
\end{equation}
under the embedding $F\mapsto P_F$ of the Grassmann manifold into $\mathcal L(\ell_2^N,\ell_2^N)$, where $P_E, P_F$ stand for the orthogonal projections onto $E$ and $F$, respectively. Hence, we can and will think of the Grassmann manifold $G_{N,k}$ as the set of orthogonal projections on $\ell_2^N$ of given rank $k$.
A typical example of such an induced metric on the Grassmannian $G_{N,k}$ 
is the one induced by the operator norm.
Note that, since we identify $G_{N,k}$ with the orthogonal projections on $\ell_2^N$ with given rank $k$,
the singular value decomposition shows that $k^{-1/p}P_E \in B_p^N$ for each $E\in G_{N,k}$. Therefore, 
\[
k^{-1/p}G_{N,k} \subseteq B_p^N \subseteq B_{\infty}^N.
\]

The following result shows that the volume of metric balls in the Grassmannian is essentially independent of the chosen unitarily invariant norm. It is an immediate consequence of Szarek's result from \cite{Sz} on nets in Grassmann manifolds.
\begin{lemma}[\cite{PV}, Corollary 2.2]\label{lem:volume metric balls grassmann}
	Let $1\leq k \leq N-1$, $\dis(\cdot,\cdot)$ be a metric on $G_{N,k}$ induced by some unitarily invariant norm on $\mathcal L(\ell_2^N,\ell_2^N)$ normalized such that $\diam(G_{N,k},\dis)=1$. Let $0<\delta < 1$. Then, for any $F\in G_{N,k}$, 
	\[
	\Big(\frac{\delta}{c_1}\Big)^{k(N-k)} \leq \mu_{N,k}\big(B_{\dis}(F,\delta)\big) \leq 	\Big(\frac{\delta}{c_2}\Big)^{k(N-k)},
	\]
	where $c_1,c_2\in(0,\infty)$ are absolute constants and $B_{\dis}(F,\delta):=\{E\in G_{N,k}\,:\, \dis(F,E) < \delta \}$.
\end{lemma}

We will need this lemma in the particular case that $\dis(\cdot,\cdot)$ is induced by the Schatten $q$-norm. If $1\le k \le N/2$ and
$E,F \in G_{N,k}$ are orthogonal, we have
\[
\dis(E,F):= \|P_E-P_F\|_{\Sc_q^N} = (2k)^{1/q}.
\]  
On the other hand, the triangle inequality shows that $\dis(E,F) \le 2 k^{1/q}$. Hence, the diameter $d$ of $G_{N,k}$ with respect to the Schatten $q$-norm satisfies $(2k)^{1/q} \le d \le 2 k^{1/q}$ and Lemma \ref{lem:volume metric balls grassmann} implies
\begin{lemma}\label{lem:volume metric balls grassmann schatten}
	Let $1\leq k \leq N/2$, $q\ge 1$ and $0<\delta < k^{1/q}$. Then, for any $F\in G_{N,k}$, 
	\[
	\Big(\frac{\delta}{c_1 k^{1/q}}\Big)^{k(N-k)} \leq \mu_{N,k}\big(B(F,\delta)\big) \leq 	\Big(\frac{\delta}{c_2 k^{1/q}}\Big)^{k(N-k)},
	\]
	with absolute constants $c_1,c_2\in(0,\infty)$ and $B(F,\delta):=\{E\in G_{N,k}\,:\, \|P_E-P_F\|_{\Sc_q^N} < \delta \}$.
\end{lemma}

\subsection{Entropy numbers and other $s$-numbers}

Let $X$ be a quasi-Banach space and let $n\in\N$. The $n$-th dyadic entropy number of $A\subseteq X$, $e_n(A,X)$, is defined as
\[
e_n(A) := e_n(A,X) := \inf\Big\{\varepsilon>0 \,:\,\exists\, x_1\dots,x_{2^{n-1}}\in X:\, A \subseteq \bigcup_{j=1}^{2^{n-1}}\big(x_j+\varepsilon B_X\big) \Big\}.
\]
For a bounded linear operator $T$ between quasi-Banach spaces $X$ and $Y$, the $n$-th dyadic entropy number, $e_n(T)$, is defined as 
\[
e_n(T) := e_n\big(T(B_X),Y\big)=\inf\Big\{\varepsilon>0 \,:\, \exists\,y_1,\dots,y_{2^{n-1}}\in Y:\, T(B_X) \subseteq \bigcup_{j=1}^{2^{n-1}}\big(y_j+\varepsilon B_Y\big) \Big\}.
\]
The dyadic entropy numbers are almost $s$-numbers and hence satisfy the following properties that we will use frequently in this text without always referring to them:
\begin{enumerate}
	\item $\|T\|\geq e_1(T) \geq e_2(T) \geq \dots \geq 0$ for $T\in\mathcal L(X,Y)$.
	\item $e_n(S+T) \leq e_n(S) + \|T\|$ for all $n\in\N$ and $S,T\in\mathcal L(X,Y)$.
	\item $e_n(RTS) \leq \|R\|\, e_n(T)\,\|S\|$ for all $n\in\N$, $S\in\mathcal L(W,X)$, $T\in\mathcal L(X,Y)$, and $R\in\mathcal L(Y,Z)$.
	\item $e_{n+m-1}(S+T) \leq C_Y\big(e_n(S)+e_m(T)\big)$ for all $n,m\in\N$ and $S,T\in\mathcal L(X,Y) $.
	\item $e_{n+m-1}(ST) \leq e_n(S)\cdot e_m(T)$ for all $n,m\in\N$ and $T\in\mathcal L(X,Y)$, $S\in\mathcal L(Y,Z)$.
\end{enumerate}
The reader may find the proofs of these facts, for instance, in \cite{CaSt} or \cite{ET}.

The following result is an analogue of \cite[Proposition 12.1.13]{Pietsch} for entropy numbers in the quasi-Banach space case.

\begin{lemma}[\cite{HKV}, Lemma 2.1]\label{lem:entropy_finite}
	Let $0<p\leq 1$, $m\in \N$, and let $X$ be a real $m$-dimensional $p$-Banach space. Then, for all $n\in\N$,
	\begin{equation}
	e_n(X\hr X) \leq 4^{1/p}\cdot 2^{-\frac{n-1}{m}}.
	\end{equation}
\end{lemma} 

We will need Carl's inequality for the case of quasi-Banach spaces, which was recently proved in \cite{HKV}. Before we state the result recall that if $X,Y$ are quasi-Banach spaces and $T\in\mathcal L(X,Y)$, then the Gelfand numbers $c_n(T)$, $n\in\N$ are defined as
\[
c_n(T):= \inf_{M\subseteq X \atop \codi(M) <n}\sup_{x\in M\atop \|x\|_X\leq 1}\|Tx\|_Y.
\]
\begin{lemma} [\cite{HKV}, Theorem 1.1] \label{lem:carl} 
Let $p,\alpha>0$. Then there exists a constant $\gamma_{\alpha,p}\in(0,\infty)$ such that, for all $p$-Banach spaces $X,Y$ and any $T\in\mathcal L(X,Y)$,
\[
\sup_{1\leq k \leq n}k^\alpha e_k(T) \leq \gamma_{\alpha,p} \sup_{1\leq k \leq n}k^\alpha c_k(T).
\]
\end{lemma}

We will also essentially need the following recent result on Gelfand numbers of identities between Schatten classes (cf. also \cite{CD}).
\begin{lemma}[\cite{CK}, Lemma 2.1]\label{lem:gelfand}
	Let $0<p\le 1$, $p \le q \le 2$. Then, for all $n,N\in\N$ with $1\le n\le N^2$,
\begin{equation}\label{eq:GelfandSchattig}
c_n\big(\Sc_p^N\hr \Sc_q^N\big)\asymp_{p,q} \min\Bigl\{1, \frac{N}{n}\Bigr\}^{1/p-1/q}.
\end{equation}
\end{lemma} 

Another type of $s$-numbers we will use are the Kolmogorov numbers. Given two quasi-Banach spaces $X$,$Y$ and $T\in\mathcal L(X,Y)$, the Kolmogorov numbers $d_n$, $n\in\N$ are defined as
\[
d_n(T):=\inf_{N\subseteq Y\atop \dim(N)<n}\,\sup_{\|x\|_X\leq 1}\,\inf_{z\in N}\|Tx-z\|_Y.
\]

\subsection{Interpolation properties of entropy numbers}

To prove the estimates in Theorem \ref{thm:main}, we need certain interpolation properties of the entropy numbers of a bounded linear map acting between quasi-Banach spaces, where one end point is fixed (see, for instance, \cite[Chapter 1.3]{ET}). Recall that a couple of quasi-Banach spaces $(X_0,X_1)$ is called an interpolation couple if both $X_0$ and $X_1$ are continuously embedded in a suitable quasi-Banach space.

\begin{lemma}[\cite{ET}, Theorem 1.3.2 (i)] \label{lem:interpol target}
	Let $X$ be a quasi-Banach space and $(Y_0,Y_1)$ be an interpolation couple of $p$-Banach spaces. Let $0<\theta<1$ and $Y_\theta$ be a quasi-Banach space such that $Y_0\cap Y_1 \subseteq Y_\theta \subseteq Y_0+Y_1$ and 
	\[
	\| y \|_{Y_\theta} \leq \|y\|_{Y_0}^{1-\theta}\cdot \|y\|_{Y_1}^{\theta},\qquad \text{for all } y\in Y_0\cap Y_1.
	\]
	Let $T\in\mathcal L(X,Y_0\cap Y_1)$, $Y_0\cap Y_1$ being equipped with the quasi-norm $\max\{ \|y\|_{Y_0}, \|y\|_{Y_1}\}$. Then, for all $k_0,k_1\in\N$,
	\begin{align*}
	e_{k_0+k_1-1}\big(T:X\to Y_{\theta}\big) & \leq 2^{1/p}\cdot e_{k_0}^{1-\theta}\big(T:X\to Y_{0}\big)\cdot e_{k_1}^{\theta}\big(T:X\to Y_{1}\big).
	\end{align*}
\end{lemma}

Since Schatten $p$-norms satisfy H\"older's inequality, we immediately obtain the following corollary.

\begin{cor}\label{cor:interpol target}
 Let $N\in\N$, $0<q_0,q_1,q \le \infty$ and $0<\theta<1$ be such that $\frac1q = \frac{1-\theta}{q_0} + \frac{\theta}{q_1}$.
 Then, for all $k_0,k_1\in\N$,
	\begin{align*}
	e_{k_0+k_1-1}\big(\Sc_p^N \hr \Sc_q^N) & \leq 2^{1/\min\{q_0,q_1,1\}}\cdot e_{k_0}^{1-\theta}\big(\Sc_p^N \hr \Sc_{q_0}^N\big)\cdot e_{k_1}^{\theta}\big(\Sc_p^N \hr \Sc_{q_1}^N\big).
	\end{align*}
\end{cor}

To formulate the analogue result for interpolation in the domain space, we recall that the $K$-functional for an interpolation couple $(X_0,X_1)$ is given by
\[
 K(t,x) = \inf \bigl\{ \| x_0 \|_{X_0} + t \|x_1\|_{X_1} \, :\, x=x_0+x_1,\, x_0\in X_0,\, x_1 \in X_1\bigr\}.
\]

\begin{lemma}[\cite{ET}, Theorem 1.3.2 (ii)] \label{lem:interpol domain}
	Let $(X_0,X_1)$ be an interpolation couple of quasi-Banach spaces and let $Y$ be a $p$-Banach space. Let $0<\theta<1$ and let $X$ be a quasi-Banach space such that $X\subseteq X_0+X_1$ and
	\[
	t^{-\theta}K(t,x) \leq \|x\|_X,\qquad \text{for all }x\in X,~ t\in(0,\infty).
	\] 
	Let $T:X_0+X_1 \to Y$ be linear and such that its restrictions to $X_0$ and $X_1$ are continuous. Then its restriction to $X$ is also continuous and, for all $k_0,k_1\in\N$,
	\begin{align*}
	e_{k_0+k_1-1}\big(T:X\to Y\big) & \leq 2^{1/p}\cdot e_{k_0}^{1-\theta}\big(T:X_0\to Y\big)\cdot e_{k_1}^{\theta}\big(T:X_1\to Y\big).
	\end{align*}
\end{lemma}

Using the singular value decomposition and the corresponding $K$-functional estimate for $\ell_p^N$-norms it is easy to see that, for $X_i=\Sc_{p_i}^N$ with $i\in\{0,1\}$ and $X=\Sc_p^N$ with  $\frac1p = \frac{1-\theta}{p_0} + \frac{\theta}{p_1}$, the assumption of Lemma \ref{lem:interpol domain} is satisfied up to a constant only depending on the parameters $p_0,p_1,\theta$.
\begin{cor}\label{cor:interpol domain}
 Let $N\in\N$, $0<p_0,p_1,p \le \infty$ and $0<\theta<1$ be such that $\frac1p = \frac{1-\theta}{p_0} + \frac{\theta}{p_1}$.
 Then, for all $k_0,k_1\in\N$,
	\begin{align*}
	e_{k_0+k_1-1}\big(\Sc_p^N \hr \Sc_q^N) & \leq c_{p_0,p_1,\theta,q} \cdot e_{k_0}^{1-\theta}\big(\Sc_{p_0}^N \hr \Sc_q^N\big)\cdot e_{k_1}^{\theta}\big(\Sc_{p_1}^N \hr \Sc_q^N\big).
	\end{align*}
\end{cor}
Lemma \ref{lem:interpol target} and Lemma \ref{lem:interpol domain} show that entropy numbers behave well with respect to interpolation
in the domain as well as in the target space. On the other hand, by the famous result of Edmunds and Netrusov \cite{EN},
this is not the case when interpolating simultaneously in the domain and in the target space.
We mention also that interpolation properties for Schatten classes in the Banach space case were already established in \cite{Pi71}. For generalizations to non-commutative $L_p$-spaces, see also \cite{PiXu}.

\section{Proof of Theorem \ref{thm:main}} \label{s:3} 

In this section, we present the proof of Theorem \ref{thm:main}.
We subdivide it into several cases, which are presented in the subsections below. These cases cover all the upper and lower bounds in Theorem \ref{thm:main}. 

\subsection{Upper bound in the case   $0<q \leq p \leq \infty$}

Here we use factorization and the entropy estimate for finite dimensional spaces given in Lemma \ref{lem:entropy_finite}. First,  Lemma \ref{lem:entropy_finite} shows that for all $n\in N$ and all $p>0$
\[
e_n\big(\Sc_p^N \hr \Sc_p^N \big)\le C_p 2^{-\frac{n-1}{N^2}}=C_p 2^{-\frac{n}{N^2}}2^{\frac{1}{N^2}}\le 2C_p2^{-\frac{n}{N^2}},
\]
with $C_p=4^{1/\min\{1,p\}}$. This is already the required bound for $p=q$.
Now we use the factorization 
$\Sc_p^N \hr \Sc_p^N \hr \Sc_q^N$ and obtain
\[
e_n\big(\Sc_p^N\hr \Sc_q^N \big) 
\le 
e_n\big(\Sc_p^N\hr \Sc_p^N\big)\cdot \|\Sc_p^N\hr \Sc_q^N\| 
\le 
2C_p2^{-\frac{n}{N^2}}N^{1/q-1/p}.
\]
This shows, up to constant, the optimal upper bound in the full range.

\subsection{Lower bound in the cases $0<q \leq p \leq \infty$ and $0<p \leq q \leq \infty, N^2 \le n$}

These bounds follow from a volume comparison argument.
Let us assume that
$B_p^N$ is covered by $2^{n-1}$ translates of the ball $\varepsilon\cdot B^N_q$, $\varepsilon>0$. Then
\begin{align*}
\vol(B_p^N)\le 2^{n-1}\vol(\varepsilon\cdot B^N_q)=2^{n-1}\varepsilon^{N^2}\vol(B^N_q)
\end{align*}
and the volume estimates for the unit balls of Schatten classes in Lemma \ref{lem:schatten-volume} imply
\begin{align*}
c_{p,q}N^{1/q-1/p}\le \left(\frac{\vol(B_p^N)}{\vol(B_q^N)}\right)^{1/N^2}\le 2^{(n-1)/N^2}\varepsilon.
\end{align*}
Hence, by the very definition of the entropy numbers,
\[
e_n\big(\Sc_p^N \hr \Sc_q^N \big)\ge c_{p,q} 2^{-\frac{n-1}{N^2}}N^{1/q-1/p}\ge c_{p,q} 2^{-\frac{n}{N^2}}N^{1/q-1/p}.
\]
This establishes the lower bound in the case $0<q\leq p \leq \infty$ for the full range of $n,N\in\N$ and the estimate is also of the right order when $0<p\leq q \leq \infty$
and $N^2 \leq n$.

\subsection{Upper bound in the case $0<p \leq q \leq \infty, N^2 \le n$}

The proof for this bound is again based on a standard volume argument. In the following, let  $\bar{r}:=\min\{1,r\}$ for some $0< r \leq \infty$. We define
\[
\varepsilon = c_{p,q}2^{-\frac{n}{N^2}}N^{1/q-1/p},
\] 
where $c_{p,q}>0$ is some constant to be specified later. We choose a maximal $\varepsilon$-net $\mathcal N\subseteq B_p^N$ in the $\|\cdot\|_{\Sc_q^N}$ sense, i.e., for all $A,B\in\mathcal N$ with $A\neq B$, we have $\| A-B\|_{\Sc_q^N}>\varepsilon$. 
Then,
\[
B_p^N \subseteq \bigcup_{A\in\mathcal N} \big(A+\varepsilon B_q^N\big).
\]
For $n\geq \lambda_{p,q}N^2$, where $\lambda_{p,q}\in(0,\infty)$ is some constant depending only on $p$ and $q$, we will show that
\begin{align}\label{eq:cardinality net}
|\mathcal N|\leq 2^{n-1}.
\end{align} 
 This will imply the desired bound for $e_n$. First, we observe that the $\bar q$-triangle inequality implies that the balls
\[
A+\frac{\varepsilon}{2^{1/\bar{q}}} B_q^N,\quad A\in\mathcal N
\]
are disjoint. Since $B_q^N\subseteq N^{1/p-1/q}B_p^N$, it holds for all $A\in\mathcal N$ that
\begin{align*}
A+\frac{\varepsilon}{2^{1/\bar{q}}} B_q^N & \subseteq B_p^N + \frac{\varepsilon}{2^{1/\bar{q}}} B_q^N  \subseteq B_p^N + \frac{\varepsilon N^{1/p-1/q}}{2^{1/\bar{q}}} B_p^N \cr
& = B_p^N + \frac{c_{p,q}2^{-\frac{n}{N^2}}}{2^{1/\bar{q}}} B_p^N 
\subseteq \gamma\cdot B_p^N,
\end{align*}
where, in the last inclusion, we used the $\bar p$-triangle inequality, and defined
\[
\gamma:=\gamma(p,q,n,N):= \Big(1+ c_{p,q}^{\bar p}2^{(-\frac{n}{N^2}-\frac{1}{\bar q})\bar p}\Big)^{1/\bar p}.
\]
By volume comparison, we conclude
\begin{align*}
|\mathcal N| \cdot \vol\Big(\frac{\varepsilon}{2^{1/\bar q}}B_q^N\Big) \leq \vol\Big(\gamma\, B_p^N\Big). 
\end{align*}
Taking the $N^2$-root, using Lemma \ref{lem:schatten-volume} and the definitions of $\varepsilon$ and $\gamma$, we obtain
\begin{align*}
|\mathcal N|^{1/N^2} \leq \frac{\alpha_{p,q}}{c_{p,q}}\Big(2^{(1/\bar q +n/N^2)\bar p}+ c_{p,q}^{\bar p}\Big)^{1/\bar p},
\end{align*}
where $\alpha_{p,q}\in(0,\infty)$ depends on the constants in the volume estimates in Lemma \ref{lem:schatten-volume}. To show \eqref{eq:cardinality net}, we have to choose $c_{p,q},\lambda_{p,q}\in(0,\infty)$ such that
\[
\frac{\alpha_{p,q}}{c_{p,q}}\Big(2^{(1/\bar q +n/N^2)\bar p}+ c_{p,q}^{\bar p}\Big)^{1/\bar p} \leq 2^{\frac{n-1}{N^2}},
\]
which is equivalent to
\begin{align}\label{eq:net estimate}
\alpha_{p,q}^{\bar p}2^{\bar p/\bar q} \leq c_{p,q}^{\bar p}\Big(2^{-\bar p/N^2}-\alpha_{p,q}^{\bar p}2^{-\bar pn/N^2}\Big).
\end{align}
Obviously, $2^{-\bar p/N^2} \geq 2^{-\bar p}$ and, for suitable $\lambda_{p,q}\in(0,\infty)$, we have $\alpha_{p,q}^{\bar{p}}2^{-\bar{p} n/N^2} \leq 4^{-\bar p}$. Therefore, we can choose a suitably large constant $c_{p,q}\in(0,\infty)$ such that the inequality \eqref{eq:net estimate} is satisfied. This completes the proof for $n\geq \lambda_{p,q}N^2$. To close the gap left, observe that, for $N^2 \leq n \leq \lambda_{p,q}N^2$, the monotonicity of the entropy numbers and the upper estimate for $e_{N^2}\big(\Sc_p^N\hr \Sc_q^N\big)$ (see next subsection), gives
\[
e_{n}\big(\Sc_p^N\hr \Sc_q^N\big) \leq e_{N^2}\big(\Sc_p^N\hr \Sc_q^N\big) \leq c_{p,q}^{'}N^{1/q-1/p} \leq c_{p,q}^{''} 2^{-n/N^2}N^{1/q-1/p},
\]
where $c_{p,q}^{'}, c_{p,q}^{''}\in(0,\infty)$.

\subsection{Upper bound in the case $0<p \leq q \leq \infty, 1 \le n \le N^2$}

For $1\le n \le 2N$, this bound follows trivially from 
\[ 
e_n\big(\Sc_p^N \hr \Sc_q^N \big) \le \| \Sc_p^N \hr \Sc_q^N \| = 1.
\]
For $2N \le n \le N^2$ we resort to the Gelfand number estimate in Lemma
\ref{lem:gelfand}, the Carl inequality in Lemma \ref{lem:carl} to transfer this into an entropy estimate and, finally, to interpolation to cover the remaining cases for $p$ and $q$.

We will show that, for all $N \leq n \leq N^2$,
\[
e_{2n}\big(\Sc_p^N\hr \Sc_q^N\big) \leq c_{p,q}\Big(\frac{N}{n}\Big)^{1/p-1/q}.
\] 
The gaps for odd $n$ can be filled as in the last subsection, by suitably increasing the constant. To prove the desired estimate in the range $N \leq n \leq N^2$, we consider four different cases.
\vskip 1mm
\noindent\textbf{Case $0<p\leq 1$ and $p\leq q \leq 2$:}
Recall that, by Lemma \ref{lem:gelfand}, 
\[
c_n(\Sc_p^N\hr \Sc_q^N) \leq c_{p,q} \Big(\frac{N}{n}\Big)^{1/p-1/q}.
\]
Therefore, by Carl's inequality (see Lemma \ref{lem:carl}) used with $\alpha=1/p-1/q$ there, we obtain
\[
n^{1/p-1/q}e_n\big(\Sc_p^N\hr \Sc_q^N\big) \leq c_{p,q} \,\sup_{k\leq n}\,\min\{k,N \}^{1/p-1/q} = c_{p,q}\,N^{1/p-1/q}.
\]
Hence,
\begin{align}\label{eq:e_n p<1 q<2}
e_n\big(\Sc_p^N\hr \Sc_q^N\big) \leq c_{p,q}\Big(\frac{N}{n}\Big)^{1/p-1/q}.
\end{align}
\vskip 1mm
\noindent\textbf{Case $0<p\leq 1$ and $q=\infty$:}
Factorizing $\Sc_p^N\hr \Sc_\infty^N $ through $\Sc_2^N$, we obtain
\[
e_{2n}\big(\Sc_p^N\hr \Sc_\infty^N\big) \leq e_n\big(\Sc_p^N\hr \Sc_2^N\big)\cdot e_n\big(\Sc_2^N \hr \Sc_\infty^N\big).
\]
From \eqref{eq:e_n p<1 q<2}, we get
\begin{align}\label{eq:before duality}
e_{2n}\big(\Sc_p^N\hr \Sc_\infty^N\big) \leq c_{p}\Big(\frac{N}{n}\Big)^{1/p-1/2}\cdot e_n\big(\Sc_2^N \hr \Sc_\infty^N\big).
\end{align}
To estimate the second factor in \eqref{eq:before duality}, we use the duality of Gelfand and Kolmogorov numbers and again Lemma \ref{lem:gelfand}. This gives
\[
d_n\big(\Sc_2^N \hr \Sc_\infty^N\big) = c_n\big(\Sc_1^N \hr \Sc_2^N\big) \leq c\Big( \frac{N}{n}\Big)^{1/2}.
\] 
Again, using Carl's inequality with Kolmogorov numbers, we obtain
\[
e_n\big(\Sc_2^N \hr \Sc_\infty^N\big) \leq c\Big( \frac{N}{n}\Big)^{1/2}.
\]
Together with the estimate \eqref{eq:before duality}, this yields
\begin{align}\label{eq: e_2n p to infinity}
e_{2n}\big(\Sc_p^N\hr \Sc_\infty^N\big) \leq c_{p}\Big(\frac{N}{n}\Big)^{1/p}.
\end{align}
This completes the proof in this case.
\vskip 1mm
\noindent\textbf{Case $0<p\leq 1$ and $p< q\leq \infty$:}
We now use interpolation in the target space (see Corollary \ref{cor:interpol target}) and obtain the estimate
\[
e_{2n}\big(\Sc_p^N\hr \Sc_q^N\big) \leq 2^{1/\bar{q}}\|\Sc_p^N\hr \Sc_p^N\|^{1-\theta}\cdot e_{2n}\big(\Sc_p^N\hr \Sc_\infty^N\big)^\theta \leq 2^{1/\bar{q}}e_{2n}\big(\Sc_p^N\hr \Sc_\infty^N\big)^\theta,
\]
with $\frac{1-\theta}{p}+\frac{\theta}{\infty}=\frac{1}{q}$. Using the upper bound in \eqref{eq: e_2n p to infinity}, we have
\begin{align}\label{eq: case p<1 and p<q<infinity}
e_{2n}\big(\Sc_p^N\hr \Sc_q^N\big) \leq c_{p,q}\Big(\frac{N}{n}\Big)^{\theta/p}=c_{p,q}\Big(\frac{N}{n}\Big)^{1/p-1/q}.
\end{align}
\vskip 1mm
\noindent\textbf{Case $1<p< q\leq \infty$:}
Here we use interpolation in the domain space (see Corollary \ref{cor:interpol domain}. Then,
\begin{align}
e_{2n}\big(\Sc_p^N\hr \Sc_q^N\big) \nonumber& \leq c_{p,q}\cdot e_{2n}\big(\Sc_1^N\hr \Sc_q^N\big)^{1-\theta}\|\Sc_q^N\hr \Sc_q^N\|^\theta \\
& \leq c_{p,q}\cdot e_{2n}\big(\Sc_1^N\hr \Sc_q^N\big)^{1-\theta},
\end{align}
where $\frac{1-\theta}{1}+\frac{\theta}{q}=\frac{1}{p}$ and with $c_{p,q}\in(0,\infty)$ being the constant from Corollary \ref{cor:interpol domain}.
Using the estimate obtained in the previous case (see inequality \eqref{eq: case p<1 and p<q<infinity}) with $p=1$ there, we get
\[
e_{2n}\big(\Sc_p^N\hr \Sc_q^N\big)\leq c_{p,q} \Big(\frac{N}{n}\Big)^{(1-1/q)(1-\theta)} = c_{p,q}\Big(\frac{N}{n}\Big)^{1/p-1/q}.
\]
This completes the proof for all $0<p\leq q \leq \infty$.

\subsection{Lower bound in the case $0<p \leq q \leq \infty, 1 \le n \le N^2$}

Note that the lower bound for $1\le n \le N$ trivially follows if the lower bound for $N\le n \le N^2$ is proved. So we assume $N \le n \le N^2$.

Let $1\leq k \leq N/2$, $q\ge 1$ and $0<\delta < k^{1/q}$. It follows from Lemma \ref{lem:volume metric balls grassmann schatten} that, for all $F\in G_{N,k}$, 
\[
\mu_{N,k}\big(B(F,\delta)\big) \leq \Big(\frac{\delta}{ck^{1/q}}\Big)^{k(N-k)}
\]
for some absolute constant $c\in(0,\infty)$, where $B(F,\delta):=\{E\in G_{N,k}\,:\, \|P_E-P_F\|_{\Sc_q^N} < \delta \}$.
We may assume that $c<4$ and choose $\delta=\frac{ck^{1/q}}{4} < k^{1/q}$. We conclude that
\[
\mu_{N,k}\big(B(F,\delta)\big) \leq 4^{-k(N-k)}=2^{-2k(N-k)} \leq 2^{-kN}.
\]
It follows that there exist a set ${\cal N} \subseteq G_{N,k}$ with cardinality $2^{kN}$ 
such that $\|P_E-  P_F\|_{\Sc_q^N} \geq \delta$ or, equivalently,
\[
 \|k^{-1/p} P_E- k^{-1/p} P_F\|_{\Sc_q^N} \geq k^{-1/p}\delta = \frac{ck^{1/q-1/p}}{4},
\] 
for distinct $E,F \in {\cal N}$. Since $k^{-1/p} P_E \in B_p^N$ for $E\in G_{N,k}$, we obtain
\[
e_{kN}\big(\Sc_p^N\hr \Sc_q^N\big)\geq \frac{ck^{1/q-1/p}}{8}.
\]
This shows the required lower bound in the case $n=kN$ and $0<p\le q \le \infty$ and $q\ge 1$. For $(k-1)N \le n \le kN$, the lower bound follows from the monotonicity of entropy numbers. The same is true for $N^2/2\le n \le N^2$ taking into account the already proved lower bound for $n=N^2$. 

The case $0<p \leq q <1$ follows by factorization,
\begin{align*}
e_{2n}\big(\Sc_p^N\hr \Sc_1^N\big) & \leq e_{n}\big(\Sc_p^N\hr \Sc_q^N\big)\cdot e_{n}\big(\Sc_q^N\hr \Sc_1^N\big).
\end{align*}
Using the bounds we have already obtained (upper and lower), we derive the lower bound for $e_{n}\big(\Sc_p^N\hr \Sc_q^N\big)$.

\section{Construction of Nets}
\label{s:4}
Let $K,N\in\N$ with $K\le N$.
We denote by $V^N_K$ the Stiefel manifold of $N\times K$ real matrices with orthonormal columns, i.e., if $\id_{\R^K}$ is the identity matrix in $\R^{K\times K}$, then 
\begin{equation}
V^N_K=\big\{U\in\R^{N\times K}\,:\,U^TU=\id_{\R^K}\big\}.
\end{equation}
We will denote the dimension of the Stiefel manifold $V_K^N$ by $d_{N,K}=K(N-(K+1)/2)$. Furthermore, if $M\subseteq X$ is a subset
of a quasi-Banach space $(X,\|\cdot\|_X)$ and $\varepsilon>0$ is a real number, we denote by $N(M,X,\varepsilon)$ (or $N(M,X,\varepsilon)$ if we want to specify the norm)
the covering number of $M$ in $X$, i.e., the minimal number of (closed) $\varepsilon$-balls in $X$ covering $M$.
Moreover, $\|U\|_{op}$ stands for the operator norm of a rectangular matrix $U\in\R^{N\times K}$.
If $K=N$, then $\|U\|_{op}=\|U\|_{\Sc_\infty^N}$. The following lemma provides an upper bound on the covering number of the Stiefel manifold in $\R^{N\times K}$ with respect to the operator norm.

\begin{lemma}\label{lem:nets:1} Let $0<\varepsilon<1$ and let $K,N\in\N$ with $K\le N$. Then
$$
N\big(V^N_K,\|\cdot\|_{op},\varepsilon\big)\le \Big(\frac{c}{\varepsilon}\Big)^{d_{N,K}},
$$
where $c\in(0,\infty)$ is a universal constant.
\end{lemma}
\begin{proof}
Let $\mathcal N_G\subseteq G_{N,K}$ be an $\varepsilon>0$ net of $G_{N,K}$ in the operator norm.
Then, to every $F\in G_{N,K}$ there is $E\in \mathcal N_G$ with $\|P_E-P_F\|_{op}<\varepsilon$.
For every $E\in \mathcal N_G$ let us fix an orthonormal basis $(u^E_1,\dots,u^E_K)$ of $E$
and denote by $U_E$ the $N\times K$ matrix with columns $u^E_1,\dots,u^E_K$. In that notation, we can write $P_E=U_E  U_E^T.$

We denote by ${\mathcal O}(K)$ the group of $K\times K$ orthogonal matrices and 
let $\mathcal N_{\mathcal O}\subseteq {\mathcal O}(K)$ be an $\varepsilon>0$ net
of ${\mathcal O}(K)$ in the operator norm. This means that to every $W\in{\mathcal O}(K)$ there is $Z\in\mathcal N_{\mathcal O}$
with $\|W-Z\|_{op}<\varepsilon.$ We put
\begin{equation}
{\mathcal N}=\{U_EZ\,:\,E\in\mathcal N_G,Z\in\mathcal N_{\mathcal O}\}\subseteq \R^{N\times K}.
\end{equation}
First, let us observe that for any $U_EZ\in \mathcal N$, 
\[
(U_EZ)^T(U_EZ)=Z^TU_E^TU_EZ=Z^T\id_{\R^K}Z=Z^TZ=\id_{\R^K},
\]
which implies that ${\mathcal N}\subseteq V_K^N$.
We will now show that ${\mathcal N}$ is a $3\varepsilon$-net of $V^N_K$ in the operator norm.

Indeed, let $U\in V_K^N$ be arbitrary and let us denote by $F$ its column space. Then there is an $E\in\mathcal N_G$
with $\|P_E-P_F\|_{op}<\varepsilon$. Also, we can write
\begin{align*}
U=P_FU=(P_F-P_E)U+P_EU=(P_F-P_E)U+U_EU_E^TU.
\end{align*}
It is now left to show that there is a $Z\in\mathcal N_{\mathcal O}$ with $\|U_E^TU-Z\|_{op}<2\varepsilon$, because then we obtain
\begin{align*}
\|U-U_EZ\|_{op}&=\|(P_F-P_E)U+U_EU_E^TU-U_EZ\|_{op}\\
&\le \|(P_F-P_E)U\|_{op}+\|U_E(U_E^TU-Z)\|_{op}\\
&\le \|P_F-P_E\|_{op}\cdot\|U\|_{op}+\|U_E\|_{op}\cdot\|U_E^TU-Z\|_{op}<3\varepsilon.
\end{align*}
We show the existence of $Z\in \mathcal N_{\mathcal O}$ with $\|U_E^TU-Z\|_{op}<2\varepsilon$.
For that sake, it is sufficient to prove that ${\rm d}(U_E^TU,{\mathcal O}(K),\|\cdot\|_{op})<\varepsilon$.
By considering the singular value decomposition of $U_E^TU$, it is even enough to show that the singular values of $U_E^TU$
lie in the interval $(1-\varepsilon,1+\varepsilon)$. Since we trivially have $\|U_E^TU\|_{op}\le \|U_E^T\|_{op}\cdot\|U\|_{op}\le 1$, we only need to
show that the singular values are larger than $1-\varepsilon$. Observe that
\begin{align*}
\big\{\|U_E^TUx\|_2\,:\,x\in\R^K,\,\|x\|_2=1\big\}&=\big\{\|U_E^Ty\|_2\,:\,y\in F,\,\|y\|_2=1\big\}\\
&=\big\{\|P_Ey\|_2\,:\,y\in F,\,\|y\|_2=1\big\},
\end{align*}
and that for $y\in F$ with $\|y\|_2=1$, we have
\begin{align*}
\|P_Ey\|_2&=\|P_Fy+(P_E-P_F)y\|_2\ge \|P_Fy\|_2-\|(P_E-P_F)y\|_2\\
&=\|y\|_2-\|(P_E-P_F)y\|_2 > 1-\varepsilon.
\end{align*}
This shows that ${\rm d}\big(U_E^TU,{\mathcal O}(K),\|\cdot\|_{op}\big)<\varepsilon$. 

We now use the estimates of Szarek (see \cite[Prop. 6 and Prop. 8]{Sz2}) and the identity
\begin{align*}
{\rm dim}\,G_{N,K}+{\rm dim}\,{\mathcal O}(K)& =K(N-K)+\frac{K(K-1)}{2} \cr
& = KN-\frac{K(K+1)}{2} 
 ={\rm dim}V_K^N=d_{N,K}
\end{align*}
to show that 
$$
|{\mathcal N}|\le |\mathcal N_G|\cdot|\mathcal N_{\mathcal O}|\le
\Bigl(\frac{c}{\varepsilon}\Bigr)^{{\rm dim}\,G_{N,K}}\cdot \Bigl(\frac{c}{\varepsilon}\Bigr)^{{\rm dim}\,{\mathcal O}(K)}
=\Bigl(\frac{c}{\varepsilon}\Bigr)^{d_{N,K}}.
$$
\end{proof}
The construction that we will present in the proof of Theorem \ref{thm:construction} is crucially based on a decomposition of a matrix into a sum of low-rank matrices. Therefore, let us define
\begin{equation}\label{eq:nets:1}
R^N_{K,q}=\big\{X\in\R^{N\times N}\,:\,\rank(X)\le K,\,\|X\|_{\Sc_q^N}\le 1\big\}.
\end{equation}
The next lemma estimates the covering numbers of $R^N_{K,q}$ in the Schatten $q$-classes $\Sc_q^N$.
\begin{lemma}\label{lem:nets:2}Let $0< q\le \infty$. Then there exists $c_q\in[1,\infty)$ such that for every $0<\varepsilon\le c_q$ and all $K, N\in\N$ with $K\le N$,
it holds
\begin{equation}
N\big(R^N_{K,q},\Sc^N_q,\varepsilon\big)\le \Big(\frac{c_q}{\varepsilon}\Big)^{2d_{N,K}+K}=\Big(\frac{c_q}{\varepsilon}\Big)^{K(2N-K)}.
\end{equation}
\end{lemma}
\begin{proof}
We write again $\bar q=\min(1,q)$. Moreover, let us assume that
\[
{\mathcal N}_{N,K}\subseteq V_K^N
\]
is an $\varepsilon/3^{1/\bar q}$-net of $V_K^N$ in the operator norm and that
\[
{\mathcal N}_{q}\subseteq B_{\ell_q^K}
\]
be an $\varepsilon/3^{1/\bar q}$-net of the unit ball $B_{\ell_q^K}$ in the quasi-Banach space $\ell_q^K$.
We now define
\[
{\mathcal N}_{K,q}^N=\Big\{\widetilde U\widetilde \Sigma {\widetilde V}^T\,:\,\widetilde U,\widetilde V\in {\mathcal N}_{N,K},\, \widetilde \Sigma=\diag(\widetilde \sigma),\, \widetilde \sigma\in {\mathcal N}_{q}\Big\}.
\]
Let us consider $X\in R_{K,q}^N$ with $X=U\Sigma V^T$, where $U,V\in V_K^N$ and $\Sigma=\diag(\sigma)$ with $\|\sigma\|_{\ell_q^K}\le 1$. Moreover, consider $Y=\widetilde U\widetilde \Sigma\widetilde V^T$,
where $\widetilde U,\widetilde V\in {\mathcal N}_{N,K}$ such that $\|U-\widetilde U\|_{op} <\varepsilon/3^{1/\bar q}$, $\|V-\widetilde V\|_{op}<\varepsilon/3^{1/\bar q}$
and where $\widetilde\sigma\in{\mathcal N}_{q}$
with $\|\sigma-\widetilde\sigma\|_{\ell^K_q}<\varepsilon/3^{1/\bar q}$. Then, we have $Y\in{\mathcal N}_{K,q}^N$ and obtain the estimate
\begin{align*}
\|X-Y\|^{\bar q}_{\Sc_q^N}&=\|U\Sigma V^T-\widetilde U\widetilde \Sigma \widetilde V^T\|^{\bar q}_{\Sc_q^N}\\
&\le \|U\Sigma (V^T-\widetilde V^T)\|^{\bar q}_{\Sc_q^N} + \|U(\Sigma-\widetilde\Sigma)\widetilde V^T\|^{\bar q}_{\Sc_q^N}+\|(U-\widetilde U)\widetilde \Sigma\widetilde V^T\|^{\bar q}_{\Sc_q^N}\\
&= \|\Sigma (V^T-\widetilde V^T)\|^{\bar q}_{\Sc_q^N} + \|\Sigma-\widetilde\Sigma\|^{\bar q}_{\Sc_q^N}+\|(U-\widetilde U)\widetilde \Sigma\|^{\bar q}_{\Sc_q^N}\\
&\le \|\Sigma\|^{\bar q}_{\Sc_q^N}\cdot \|V^T-\widetilde V^T\|_{op}^{\bar q} + \|\sigma-\widetilde\sigma\|^{\bar q}_{\ell_q^N}+\|U-\widetilde U\|_{op}^{\bar q}\cdot \|\widetilde \Sigma\|^{\bar q}_{\Sc_q^N}\\
&<\varepsilon^{\bar q}/3+\varepsilon^{\bar q}/3+\varepsilon^{\bar q}/3=\varepsilon^{\bar q}.
\end{align*}
Here, we have used the inequality $\|AB\|_{\Sc_q^N}\le \|A\|_{op}\cdot \|B\|_{\Sc_q^N}$ (see, for instance, \cite[Proposition IV.2.4]{Bhatia}).
Finally, we use Lemmas \ref{lem:nets:2} and \ref{lem:entropy_finite} to estimate the number of elements in ${\mathcal N}_{K,q}^N$. We obtain
\begin{align*}
|{\mathcal N}_{K,q}^N|\le |{\mathcal N}_{N,K}|\cdot |{\mathcal N}_{q}|\cdot|{\mathcal N}_{N,K}|\le \Bigl(\frac{c_q}{\varepsilon}\Bigr)^{2d_{N,K}+K}
=\Bigl(\frac{c_q}{\varepsilon}\Bigr)^{K(2N-K)}.
\end{align*}
\end{proof}

In the next theorem, we present a constructive proof for the upper bound of the entropy numbers of natural identities of Schatten classes in the range $N \leq n \leq N^2$. Our construction is crucially based on a low-rank matrix decomposition.

\begin{theo}\label{thm:construction}Let $0<p\le q\le\infty$ and $N\le n\le N^2$. Then
\begin{equation}\label{eq:est_up1}
e_n\big(\Sc_p^N\hookrightarrow \Sc_q^N\big)\le C_{p,q} \Bigl(\frac{N}{n}\Bigr)^{1/p-1/q},
\end{equation}
where $C_{p,q}\in(0,\infty)$ depends only on $p$ and $q$.
\end{theo}
\begin{proof}
We subdivide the proof into three steps.

\noindent\textbf{Step 1:} By the monotonicity of entropy numbers, it is enough to consider $N=2^\nu$ and $n=2^{\ell}\cdot 2^{\nu}$, where $1\le \ell\le \nu$ are natural numbers.
In that case, $N/n=2^{-\ell}$.

Let now $A\in \Sc_p^N$ be an arbitrary matrix with $\|A\|_{\Sc_p^N}\le 1$. We use the singular value decomposition of $A$ in the form $A=U\Sigma V^T$
to decompose it into a sum of low-rank matrices. If $\Sigma=\diag(\sigma_1,\dots,\sigma_{N})$ with $\sigma_1\ge\sigma_2\ge\dots\ge\sigma_N\ge 0$ is the diagonal matrix collecting the singular values of $A$ on its diagonal, we define
\[
A_1:=U\Sigma_1V^T,\qquad \Sigma_1:=\diag(\sigma_1,0,\dots,0)
\]
and, for $j=2,\dots,\ell$,
\[
A_j:=U\Sigma_jV^T,\qquad \Sigma_j:=\diag(0,\dots,0,\sigma_{2^{j-1}},\dots,\sigma_{2^j-1},0,\dots,0).
\]
That way, we obtain $\rank(A_j)\le 2^{j-1}$ for $j=1,\dots,\ell$.
Furthermore, $\|A_1\|_{\Sc^N_q}\le 1$ and, for all $j=2,\dots,\ell$, we have
\begin{align*}
\|A_j\|_{\Sc_q^N}&=\bigg(\sum_{u=2^{j-1}}^{2^j-1}\sigma_u^q\bigg)^{1/q}\le \Bigl(2^{j-1}\sigma_{2^{j-1}}^q\Bigr)^{1/q}=2^{(j-1)/q}\bigl(\sigma_{2^{j-1}}^p\bigr)^{1/p}\\
&\le2^{(j-1)/q}\bigg(\frac{1}{2^{j-1}}\sum_{u=1}^{2^{j-1}}\sigma_u^p\bigg)^{1/p}\le 2^{(j-1)(1/q-1/p)}.
\end{align*}
We can therefore decompose as follows,
\[
A=A_1+A_2+\dots+A_\ell+A^c,
\]
where $\rank(A_j)\le 2^{j-1}$ and $\|A_j\|_{\Sc_q^N}\le 2^{(j-1)(1/q-1/p)}$
for all $j=1,\dots,\ell$. Finally, we obtain 
\begin{align*}
\|A^c\|^q_{\Sc^N_q}&=\sum_{u=2^\ell}^{N}\sigma_{u}^q=\sum_{u=2^\ell}^{N}\sigma_{u}^p\sigma_u^{q-p}\le \sigma_{2^{\ell}}^{q-p}\cdot\sum_{u=2^\ell}^N\sigma_u^p\\
&\le \bigg(\frac{1}{2^\ell}\sum_{u=1}^{2^\ell}\sigma_u^p\bigg)^{\frac{q-p}{p}}\cdot\sum_{u=2^\ell}^N\sigma_u^p\\
&\le 2^{-\ell\cdot\frac{q-p}{p}}\|A\|_{\Sc_p^n}^{q-p}\cdot\|A\|_{\Sc_p^N}^{p}
\end{align*}
and $\|A^c\|_{\Sc_q^N}\le 2^{\ell(1/q-1/p)}.$
\vskip 1mm
\noindent\textbf{Step 2:} We first observe that, for $j=1,\dots,\ell$, the low-rank matrix $2^{(j-1)(1/p-1/q)}A_j$ belongs to $R^{N}_{2^{j-1},q}$, where $R^{N}_{2^{j-1},q}$ was defined in \eqref{eq:nets:1}.
Now let ${\mathcal N}_j\subseteq R^{N}_{2^{j-1},q}$, $j=1,\dots,\ell$
be an $\varepsilon_j$-net. Then,
$$
{\mathcal N}:=\left\{\sum_{j=1}^\ell 2^{-(j-1)(1/p-1/q)}Z_j\,:\, Z_j\in {\mathcal N}_j\ \text{for all}\ j=1,\dots,\ell\right\}
$$
is an $\varepsilon$-net of the unit ball of $\Sc_p^N$ in $\Sc_q^N$, where
\begin{equation}
\varepsilon^{\bar q}=\sum_{j=1}^\ell \big[2^{-(j-1)(1/p-1/q)}\varepsilon_j\big]^{\bar q}+2^{\ell{\bar q}(1/q-1/p)}.
\end{equation}
Furthermore, ${\mathcal N}$ has $\displaystyle |{\mathcal N}|=\prod_{j=1}^\ell|{\mathcal N}_j|$ elements.
\vskip 1mm
\noindent\textbf{Step 3:} For $j=1,\dots,\ell$, we choose $\varepsilon_j=c_{q}\cdot 2^{(j-\ell)(1/p-1/q+\alpha)}$, where $c_q\in[1,\infty)$ is the constant from Lemma \ref{lem:nets:2}
and the value of the parameter $\alpha>0$ will be specified later.
This gives
\begin{align}\label{eq:est_up11}
\notag \varepsilon^{\bar q}&=\sum_{j=1}^\ell \big[2^{-(j-1)(1/p-1/q)}\varepsilon_j\big]^{\bar q}+2^{\ell{\bar q}(1/q-1/p)}\\
\notag&=\sum_{j=1}^\ell \big[c_{q}2^{(1-\ell)(1/p-1/q)}2^{(j-\ell)\alpha}\big]^{\bar q}+2^{\ell{\bar q}(1/q-1/p)}\\
&=c_{q}^{\bar q}2^{(1-\ell){\bar q}(1/p-1/q)}\sum_{j=1}^\ell 2^{(j-l)\alpha{\bar q}}+2^{\ell{\bar q}(1/q-1/p)}\\
\notag &\le c_{\alpha,q}^{\bar q}2^{-\ell{\bar q}(1/p-1/q)}=c_{\alpha,q}^{\bar q}(N/n)^{{\bar q}(1/p-1/q)}.
\end{align}
Furthermore, by Lemma \ref{lem:nets:2},
\begin{align*}
|{\mathcal N}|&=\prod_{j=1}^\ell|{\mathcal N}_j|\le \prod_{j=1}^\ell \Bigl(\frac{c_q}{\varepsilon_j}\Bigr)^{2^{j-1}(2N-2^{j-1})}
\le \prod_{j=1}^\ell \Bigl(2^{(\ell-j)\cdot (1/p-1/q+\alpha)}\Bigr)^{2^{j}N}
\end{align*}
and
\begin{align}
\notag\log_2(|{\mathcal N}|)&=\sum_{j=1}^\ell\log_2(|{\mathcal N}_j|)\le \sum_{j=1}^\ell 2^{j}N\log_2(2^{(\ell-j)\cdot (1/p-1/q+\alpha)})\\
\label{eq:est_up10}&=(1/p-1/q+\alpha)N\sum_{j=1}^\ell 2^{j}(\ell-j)=(1/p-1/q+\alpha)N2^{\ell}\sum_{j=1}^\ell 2^{j-\ell}(\ell-j)\\
\notag&\le 2(1/p-1/q+\alpha)2^\ell N=2(1/p-1/q+\alpha)n.
\end{align}
We now choose $\alpha=1$ and $\gamma\in\N$ with $\gamma\ge 1+2(1/p-1/q+1)$.
Combining \eqref{eq:est_up11} with \eqref{eq:est_up10}, we obtain the upper bound $e_{\gamma n}\big(\Sc_p^N\hookrightarrow \Sc_q^N\big)\le c(N/n)^{1/p-1/q}$.
\end{proof}


\end{document}